\documentclass{article}
\usepackage[T1]{fontenc}
\usepackage{graphicx}
\usepackage{amsmath}
\usepackage{amssymb}
\usepackage{amsfonts}
\usepackage{amsfonts}
\usepackage{amsthm}

\newcommand{\modul}{\operatorname{mod}}
\newcommand{\vol}{\operatorname{vol}}		
\newcommand{\var}{\operatorname{Var}}	
\newcommand{\Ind}{\mathbb{I}}
\newcommand{\qq}{\mathbf{q}}	
\newcommand{\dd}{\mathbf{d}}
\newcommand{\whp}{{with high probability}}
\newcommand{\Whp}{{With high probability}}
\newcommand{\V}{\mathcal{V}}
\newcommand{\barV}{\mathcal{\tilde{V}}}
\newcommand{\W}{\mathcal{W}}
\newcommand{\Po}[1]{\textrm{Po}\left(#1\right)}
\newcommand{\Bin}[2]{\textrm{Bin}\left(#1,#2\right)}
\newcommand{\E}{\mathbb{E}}
\newcommand{\err}{\text{err}}
\newcommand{\Pra}[1]{\Pr\left(#1\right)}
\newcommand{\Gnmp}{\mathcal{G}\left(n,m,p\right)}
\newcommand{\hatGnmp}{\hat{\mathcal{G}}\left(n,m,p\right)}
\newcommand{\barGnmp}{\bar{\mathcal{G}}\left(n,m,p\right)}
\newcommand{\Gnp}{G(n,{\bf \bar p})}
\newcommand{\barp}{{\bf \bar p}}
\newcommand{\G}{\mathcal{G}}
\newcommand{\hatG}{\hat{\mathcal{G}}}
\newcommand{\eps}{\varepsilon}

\theoremstyle{plain}
\newtheorem{theorem}{Theorem}
\newtheorem{definition}[theorem]{Definition}
\newtheorem{lemma}[theorem]{Lemma}
\newtheorem{corollary}[theorem]{Corollary}
\newtheorem{remark}[theorem]{Remark}

\title{Modularity of random intersection graphs}
\author{Katarzyna Rybarczyk,\\ {\small\sl Adam Mickiewicz University, Pozna\'n}}
\date{}

\begin{document}
\maketitle	
\begin{abstract}
Modularity was introduced by Newman and Girvan in 2004 and is used as a measure of community structure of networks represented by graphs. In our work we study modularity of the random intersection graph model first considered  by Karo\'nski, Scheinerman, and Singer--Cohen in  1999. Since their introduction, random intersection graphs has attracted much attention, mostly due to their application as networks models. In our work we determine the range of 
parameters in which modularity detects well the community structure of the random intersection graphs, as well as give a range of parameters for which there is a community structure present but not revealed by modularity. We also relate modularity of the random intersection graph to the modularity of other known random graph models.  

\end{abstract}

{\bf keywords} Modularity, Random intersection graph, Random graphs, Affiliation networks, Complex networks.

\section{Introduction}\label{Introduction}

{\it Motivation. } 
Detecting and measuring the presence of community structure in the network are undoubtedly natural motivations for compelling research in the field of computer science. Mainly because identifying communities has a number of practical applications such as identifying the groups of common interests in social networks, classifying fake news, identifying proteins with the same biological functions and many others \cite{KaPrTh21,NeBook10}.
In this article we concentrate on the notion of measure of the presence of community structure in the network's graph called \textit{modularity}. Modularity was introduced by Newman and Girvan in 2004 \cite{NeBook10,NeGi04}. It relays on the idea that in the graph with a present community structure the vertex set might be  partitioned into subsets in which there are much more internal edges than we would expect by chance (see Definition~\ref{def:modularity}). Modularity is commonly used in community detection algorithms as a quality function judging the performance of the algorithms~\cite{KaPrTh21}, and also as a central ingredient of such algorithms, for example in Louvain algorithm \cite{BlGuLaLe08}, Leiden algorithm \cite{TrWaEc19} or Tel-Aviv algorithm \cite{GiSh_23}.

First theoretical results concerned modularity of deterministic graphs (see for example \cite{Bagrow_trees_12,Brandes2008,Majstorovic2014}). This line of research is still continued, see for example a recent result \cite{LaSu23} by Laso{\'n} and Sulkowska. However, in the context of studying large networks, it is crucial to determine the modularity of their natural models -- random graphs. Systematic and thorough studies on this topic have been undertaken by McDiarmid and Skerman in a series of articles \cite{McSk_reg_lattice_13,McDiSk2018RegTreelike,McDiarmidSkerman2020}. However, only recently has research begun on the modularity of these random graph models, which are considered to reflect the properties of complex networks. Here we should mention the results about modularity of the random graph on the hyperbolic plane 
  by Chellig, Fountoulakis, and Skerman \cite{ChFoSk21}, and those concerning the basic preferential attachment graph by  Prokhorenkova, Pra{\l}at, and Raigorodskii \cite{PrPrRa16_LNCS,PrPrRa17} and by Rybarczyk and Sulkowska \cite{RybSulk2025+}. 
 
 It is a plausible that many of the properties of complex networks are related to a known or hidden bipartite structure representing the relationships between network elements and their properties and the fact that alike elements (with common properties) tend to form connections \cite{GuillaumeLatapy2004Bipartite}. A theoretical model of a random graph based on a bipartite structure of vertices and attributes has been introduced by Karo\'nski, Schienerman, and Singer-Cohen in~\cite{KaronskiSchienerSinger1999Subgraph} and called the random intersection graph. The model was further generalised by Godehardt and Jaworski~\cite{GodehardtJaworski2003TwoModels}. Moreover, some other variants of graphs based on a bipartite structure have been studied since, for example \cite{BloznelisLaskela2021Superpositions,BloznelisKaronski2014RIGProcess,BloznelisLeskela2023SuperpositionsClust,HofstadKamjathyVadon2021RIG}. In the classical random intersection graphs, vertices choose their neighbours based on randomly assigned attributes they possess. It might be interpreted in such a manner that the more the elements of the network are alike, the more probable is that they are connected. For comprehensive surveys on the classical results concerning random intersection graphs we refer the reader to \cite{Bloznelis2015SurveyModels,Bloznelis2015SurveyProperties,Spirakis2013Survey}. It turns out that with well-chosen parameters, random intersection graphs have many properties of complex networks.
 To mention just a few, random intersection graphs are known for having tunable clustering and degrees that in some ranges of parameters have the power law distribution, see for example \cite{Bloznelis2015SurveyProperties} and references therein. One of the reasons for having these properties is the fact that, unlike in Erd\H{o}s-R\'enyi graph, in random intersection graphs edges do not appear independently.   
 We stress here the fact that, in many ranges of parameters, random intersection graphs have large clustering i.e. apparently have some community structure.
 However still little is known about their modularity.

 In our work we concentrate on the modularity of the classical model of the random intersection graph introduced in \cite{KaronskiSchienerSinger1999Subgraph}. However we hope that this work will be followed by a more thorough study of other random intersection graph models. 

{\it Notation.} 
Before we give a formal definition of the considered model we introduce some basic notations. We denote by $[n]=\{1,2,\ldots,n\}$ the set of the first $n$ natural numbers. By $(n)_k=n\cdot(n-1)\cdots(n-k+1)$ we denote the falling factorial.
Moreover, in what follows, we use standard asymptotic notations $o(\cdot), O(\cdot), \Omega(\cdot), \Theta(\cdot), \sim, \asymp, \ll, \gg,$ as defined in \cite{JLR_random_graphs}. Also $\omega=\omega_n\to \infty$ will be a function tending to infinity arbitrarily slowly as $n\to \infty$. All limits are taken as $n\to \infty$ and inequalities  hold for large $n$. We say that an event $\mathcal{A}$ holds {\whp} if $\Pra{\mathcal{A}}\to 1$ as $n\to \infty$. We use $\Bin{\cdot}{\cdot}$ and $\Po{\cdot}$ to denote the binomial and Poisson distribution, resp.. For any graph $G$ we denote by $V(G)$ its vertex set and by $E(G)$ its edge set. Moreover for any subset $S$ of the vertex set, $S\subseteq V(G)$, of a graph $G$ we denote by $\bar S=V(G)\setminus S$ its complement. By $e_G(S)$ we denote the number edges in $G$ with both ends in $S$ and by $e_G(S,\bar S)$ the number of edges with exactly one end in $S$.  In addition by $\vol_G(S)=2e_G(S)+e_G(S,\bar S)$ we denote the sum of degrees of vertices from $S$ in $G$. We omit the subscript $G$, when it is clear from the context which graph $G$ we have in mind. Moreover, we write $e(G)=e_{G}(V(G))$ and $\vol(G)=\vol_{G}(V(G))$.

{\it Model and its properties. } 
In 
the binomial 
random intersection graph $\Gnmp$ introduced in~\cite{KaronskiSchienerSinger1999Subgraph} there is a set of vertices $\V=\{v_1,\ldots,v_n\}$ and an auxiliary set of attributes $\W=\{w_1,\ldots,w_m\}$.  
All vertices $v_i$, $i\in [n]$, are attributed random subsets 
$\W(v_i)$, $i\in [n]$, of $\W$ in such a manner that for every $v_i$, $i\in [n]$, each $w\in \W$ is included in $\W(v_i)$ independently at random with probability $p$. Two vertices $v_i$ and $v_j$, $i,j\in [n]$, are connected by an edge in the random intersection graph $\Gnmp$ if their attribute sets intersect, i.e. $\W(v_i)\cap \W(v_j)\neq \emptyset$. 
We consider a sequence of random intersection graphs 
$\Gnmp$, where 
$m=m_n\to+\infty$ and  $p=p_n\to 0$ as $n\to+\infty$.

Note that for each $i$, $i\in [n]$, the size of $\W(v_i)$ is the binomial random variable $\Bin{m}{p}$ with the expected value $mp$.
We may interpret the structure of $\Gnmp$ in another manner. Let $\V(w_i)=\{v\in \V: w_i \in \W(v)\}$ be the set of vertices that chose attribute $w_i$, $i\in [m]$. Then the edge set of $\Gnmp$ is the sum of edges of $m$ independent cliques on vertex sets $\V(w_i)$, $i\in [m]$, with sizes with the binomial distribution $\Bin{n}{p}$ and the expected value $np$.  

Now let us discuss some properties of $\Gnmp$ and the chosen range of parameters. 
In the random intersection graph $\Gnmp$ an edge $v_iv_j$, $i,j\in [n]$, is present with probability ${\bf \hat p} = 1-(1-p^2)^{m}= 1-e^{-mp^2+O(mp^4)}$. However edges do not appear independently.  In this article we focus on sparse random intersection graphs, i.e. when ${\bf \hat p}=o(1)$, which is equivalent to $mp^2=o(1)$. Our focus is motivated by the fact that most interesting results concerning modularity of random graphs concern the spare case. Moreover we assume that the expected number of edges tends to infinity, i.e. $n^2mp^2\to \infty$, as otherwise with positive probability $\Gnmp$ is an empty graph (see for example \cite{FillScheinSinger2000Equivalence}). In addition, we set $\dd=nmp^2\sim (n-1){\bf \hat p}$ which is asymptotically close to the average degree of $\Gnmp$. In some cases we  assume that $\dd\ge 1$. This is a natural assumption as $\dd=1$ is the phase transition threshold in the evolution of $\Gnmp$ for $m\not\asymp n$ (see \cite{Behrish2007PhaseTransition}) and $\dd\asymp 1$ is the phase transition threshold in the remaining cases $m\asymp n$ (see \cite{LagerasLindholm2008PhaseTransition2}). Our interest in the range $\dd\ge 1$ is  motivated by  known results about $G(n,{\bf \bar{p}})$, the Erd\H{o}s--R\'enyi random graph with independent edges. 

The structure of $\Gnmp$ is very diverse and a lot depends on the values $np$ and $mp$. A long line of research showed that for $m\gg n^3$ the graphs $\Gnmp$ and $\Gnp$ with ${\bf \bar p}$ close to $1-e^{-mp^2}$ are equivalent (\cite{BrennanBresterNagaraj2020Equivalence,FillScheinSinger2000Equivalence,KimLeeNa2018Equivalence,Rybarczyk2011Equivalence}). We should point out that $\Gnmp$   {\whp} is not edgeless for $p=\Omega(1/n\sqrt{m})$ and not a complete graph for $p = O(\sqrt{\ln n/m})$. Therefore $m\gg n^3$ implies $mp$ very large and $np$ very small. On the other hand, in some range of parameters, when $np\ll 1$, even though $\Gnmp$ and $\Gnp$ are not equivalent, there is some close relation between the models \cite{Rybarczyk2011Coupling,Rybarczyk2017coupling}. Nevertheless for $np\gg 1$ the models differ a lot (see for example results form \cite{KaronskiSchienerSinger1999Subgraph} and discussion in \cite{FillScheinSinger2000Equivalence}). One of the motivations for our work was comparison of and discussion about the relation between the results concerning modularity of  $\Gnmp$ and $\Gnp$.

{\it Related work on modularity.} Modularity was defined by Newman and Girvan in 2004 \cite{NeBook10,NeGi04} and has been extensively studied since.

\begin{definition}\label{def:modularity}
	Let $G$ be a graph with at least one edge. For a partition $\mathcal{A}$ of $V(G)$ define a modularity score of $G$ as
	\[
	\modul_{\mathcal{A}}(G) = \sum_{S\in\mathcal{A}}\left(\frac{e(S)}{e(G)}-\left(\frac{\vol(S)}{\vol(G)}\right)^2\right).
	\]
	Modularity of $G$ is given by
	\[
	\modul(G) = \max_{\mathcal{A}}\modul_{\mathcal{A}}(G),
	\]
	where maximum runs over all the partitions of the set $V(G)$.
\end{definition}
It follows straight forward that $\modul(G)\in [0,1]$.
It is considered that "large" modularity, closer to $1$, with some exceptions, is related to visible community structure, while modularity close to $0$ depicts no community structure.
In the context of our research the most important known results concern other random graph models. The modularity of the Erd\H{o}s--R\'enyi random graph with independent edges was studied by McDiarmid and Skerman \cite{McDiarmidSkerman2020}. In particular they showed that for the average degree $n{\bf \bar p}\le 1+o(1)$ the modularity of $\Gnp$ is tending to 1 in probability. It is consistent with other known results concerning modularity of random graphs with no communities and more tree-like structure (see \cite{McDiSk2018RegTreelike}). However we focus on the case above phase transition threshold, i.e. when $n{\bf \bar p}\ge C$, for some constant $C$, and ${\bf \bar p}$ is bounded away from $1$. McDiarmid and Skerman proved that then  {\whp} $\modul(\Gnp)\asymp 1/\sqrt{n{\bf \bar p}}$, where $n{\bf \bar p}$ is asymptotically the average degree of $\Gnp$. This result is consistent with modularity   $\Theta(1/\sqrt{r})$ for random $r$--regular graphs for $r\ge 3$ \cite{McDiSk2018RegTreelike} and modularity of the basic preferential attachment model with average degree $2h$ that is $\Omega(1/\sqrt{h})$ \cite{PrPrRa16_LNCS,PrPrRa17} and $O(\sqrt{\ln h/h})$ \cite{RybSulk2025+}. Moreover this result shows that for average degree $n{\bf \bar p}\to \infty$ as $n\to \infty$ {\whp} $\Gnp$ has modularity $o(1)$.

{\it Results.} Recall that $mp$ is the average size of the attribute set $\W(v_i)$, $i\in [n]$, of a vertex and $np$ is the average size of the clique $\V(w_i)$, $i\in [m]$, related to an attribute. It is plausible that large modularity (i.e. apparent and visible community structure) should be related to large $np$ and small $mp$. Since then communities should be defined by attributes and, in average, no vertex (element of the network) is associated with too many such communities.  
The first of the results concentrates on this case, i.e. $mp\ll 1$ and $np\gg 1$. In fact it follows from the proof that the modularity is in fact closely related to the community structure given by the cliques formed by elements of the network sharing a common attribute. 

We recall that assumption $mp^2=o(1)$ means that the graph is sparse (edge probability tends to $0$ as $n\to \infty$) and assumption $n^2mp^2\to \infty$ as $n\to\infty$ implies that {\whp} $\Gnmp$ has at least one edge.  As natural, these assumptions appear in all the theorems.
\begin{theorem}\label{Thm:Smallmp} 
	Let $mp^2=o(1)$ and $n^2mp^2\to \infty$ as $n\to\infty$.
	There exists $C>0$ such that for all $\eps>0$ there exists $A_\eps$ such that if
	$$
	npe^{-mp}\ge A_\eps,
	$$ 
	then with probability at least $1-\eps$
	$$
	\modul (\Gnmp)\ge (1-C\eps)e^{-mp}.
	$$
\end{theorem}
\begin{remark}
	From the proof it follows that $C\le 16$, however we made no attempt to make the constant optimal. 
\end{remark}
\begin{corollary}\label{Cor:Smallmp}
	Let $mp^2=o(1)$, $n^2mp^2\to \infty$, and $np\to\infty$ as $n\to\infty$. Then {\whp}
	\[
	\modul(\Gnmp)\ge (1+o(1))e^{-mp}.
	\] 
	In particular, for $mp=o(1)$
	\[
	\modul(\Gnmp)=1+o(1).
	\]
\end{corollary}
The following theorem shows that $np\to \infty$, as $n\to\infty$, does not implicitly means that $\Gnmp$ has high modularity. 
\begin{theorem}\label{Thm:Large_np_mp}
Let $mp^2=o(1)$, $n^2mp^2\to \infty$, $n=o(m)$, and $np/\ln \frac{m}{n}\to\infty$ as $n\to\infty$. Then {\whp} 
	\[
	\modul(\Gnmp)=O\left(\left(\frac{\ln(m/n)}{np}\right)^{1/2}+\frac{n}{m}+\omega mp^2\right),
	\]
for any $\omega\to \infty$ arbitrarily slowly as $n\to\infty$.
\end{theorem}
We should stress that in the above considered case not only $np\to \infty$, but also $mp\to \infty$, as $n\to \infty$.

The following two results are an attempt to relate the known results concerning $\Gnp$ to modularity of $\Gnmp$. Recall that for $\Gnp$ with the average degree $n\barp=\Omega(1)$ and $\barp=o(1)$ modularity is asymptotically $\Theta(1/\sqrt{n\barp})$. 
\begin{theorem}\label{Thm:Smallnp}
	Let $np=o(1)$, $mp^2=o(1)$, and $\dd=nmp^2\ge 1$. Then   for any $A>0$ there exists a constant $C$ such that  for any $\omega=\omega_n\to \infty$ as $n\to \infty$ {\whp} 
	\[
	\modul(\Gnmp)\le \frac{C}{\sqrt{\dd}}+O\left((np)^{A}+\omega mp^2\right).
	\]
\end{theorem}
\begin{corollary}\label{Cor:Smallnp}
	Let $m\ge n^2$, $mp^2=o(n^{-1/3})$, and $\dd=nmp^2\ge 1$.
	Then there exists a constant $C$ such that 
	\[
	\modul(\Gnmp)\le \frac{C}{\sqrt{\dd}}.
	\]
\end{corollary}
The last theorem relates both models $\Gnmp$ and $\Gnp$ more directly.
\begin{theorem}\label{Thm:LikeGnp} Let $np=o(1)$ and 
$$\barp = 1-e^{-m\hat q/\binom{n}{2}}, \text{ where }\hat{q}=1-(1-p)^n-np(1-p)^{n-1}.$$ Then for any $\delta>0$ {\whp} 
	\[
	\modul(\Gnmp)=(1+o(1))\modul(\Gnp)+O((np)^{1-\delta}).
	\]
\end{theorem}

The remaining part of the article is organised as follows. In the next sections we give auxiliary results. In particular, Section~\ref{Sec:Modularity} states one known and one new useful result about modularity. In Section~\ref{Sec:Edges} we give some general results about the edge counts in $\Gnmp$. Section~\ref{Sec:Auxiliary} states known probabilistic inequalities that are used in the proofs. Sections~\ref{Sec:Large_np_mp} to \ref{Sec:LikeGnp} present the proofs of theorems. We finish with concluding remarks. 

\section{Modularity lemmas}\label{Sec:Modularity}
We will use a result on modularity by Dinh and Thai.

\begin{lemma}[\cite{DiTh11}, Lemma 1] \label{lemma:only2parts}
	Let $G$ be a graph with at least one edge and let $k \in \mathbb{N} \setminus \{1\}$. Then
	\[
	\modul(G) \leq \frac{k}{k-1} \max_{\mathcal{A}: |\mathcal{A}| \leq k} \modul_{\mathcal{A}}(G).
	\]
	In particular,
	\[
	\modul(G) \leq 2  \max_{\mathcal{A}: |\mathcal{A}| \leq 2} \modul_{\mathcal{A}}(G).
	\]
\end{lemma}
We prove the following corollary of this lemma.
\begin{corollary}\label{corollary:OnlySbig}
	Let $G$ be a graph with at least one edge  then
	\[
	\modul(G)\le 4 \max_{S\subseteq \V, |S|\ge n/2}\left(\frac{e(S)}{e(G)}-\left(\frac{\vol(S)}{\vol(G)}\right)^2\right).
	\]
\end{corollary}	
\begin{proof}
	First note that
	\begin{equation*}\begin{split}
		e(S)&=e(\bar S)+ e(G)-\vol(\bar{S}),\\
		\vol^2(S)&=\left(\vol(G)-\vol(\bar S)\right)^2\\
		&=\vol(G)\left(\vol(G)-2\vol(\bar S)\right)+\vol^2(\bar S).
	\end{split}\end{equation*}
	Therefore
	\begin{equation}\label{Eq:ModularitySymmetry}
		\begin{split}
			&\frac{e(S)}{e(G)}-\left(\frac{\vol(S)}{\vol(G)}\right)^2\\
			&=\frac{e(\bar S)}{e(G)}+ \frac{e(G)-\vol(\bar{S})}{e(G)}-\frac{\vol(G)\left(\vol(G)-2\vol(\bar S)\right)}{\vol^2(G)}-\left(\frac{\vol(\bar S)}{\vol(G)}\right)^2\\
			&=\frac{e(\bar S)}{e(G)}-\left(\frac{\vol(\bar S)}{\vol(G)}\right)^2
		\end{split}
	\end{equation}
	Note that for $|\mathcal A|=1$ we have $\modul_{\mathcal A}(G)=0$. 
	Thus, as either $S\ge n/2$ or $\bar S> n/2$, by Definition~\ref{def:modularity}, Lemma~\ref{lemma:only2parts}, and \eqref{Eq:ModularitySymmetry}
	\begin{equation*}\begin{split}
		\modul(G)&\le 2\max_{\mathcal A, |\mathcal A|=2}\modul_{\mathcal A}(G)\\
		&=2\max_{S\subseteq \V}\left(\frac{e(S)}{e(G)}-\left(\frac{\vol(S)}{\vol(G)}\right)^2
		+\frac{e(\bar S)}{e(G)}-\left(\frac{\vol(\bar S)}{\vol(G)}\right)^2
		\right)\\
		&=2\max_{S\subseteq \V, |S|\ge n/2}2\left(\frac{e(S)}{e(G)}-\left(\frac{\vol(S)}{\vol(G)}\right)^2
		\right).
	\end{split}\end{equation*}
\end{proof}

\section{Counting edges in $\Gnmp$}\label{Sec:Edges}
Recall that, as mentioned in the introduction, in $\Gnmp$ each edge appears with probability 
\[
1-(1-p^2)^m= (1-e^{-(1+o(1))mp^2}) = mp^2(1+O(mp^2)),\text{ for }mp^2=o(1).
\]
Therefore 
\[
\E e(\Gnmp)=\binom{n}{2}mp^2(1+O(mp^2))=\frac12 n^2mp^2(1+O(mp^2+n^{-1})).
\]
Moreover, for $mp^2=o(1)$,
\[
\var (e(\Gnmp))\sim \binom{n}{2}mp^2+(n)_3mp^3\asymp n^2mp^2(1+np).
\]
By Chebyshev's inequality we get that for $mp^2\ll 1$ and $n^2mp^2\gg 1$ {\whp}
\begin{multline}\label{Eq:eGconc}
	\vol(\Gnmp)=2e(\Gnmp)\\=n^2mp^2\left(1+O\left( n^{-1}+mp^2+\omega (1+np) (n^2mp^2)^{-1/2}\right)\right) \sim n^2mp^2.
\end{multline}
Recall that $\omega=\omega_n$ is a function tending to infinity arbitrarily slowly as $n\to \infty$.

For $S\subseteq\V$, let
\begin{equation*}
	X_{i,S}=|\V(w_i)\cap S|,\quad V_i=X_{i,\V}=|\V(w_i)|,\quad i\in [m].
\end{equation*}
By definition, each attribute $w_i\in\W$, $i\in [m]$, generates a clique of size $V_i$ in $\Gnmp$. Moreover each edge of $\Gnmp$ is included in at least one clique. However some edges might be included in two or more such cliques. Let us denote by
\[
E_1=|\{vv'\in E(\Gnmp):\exists_{i,j\in [m]}v,v'\in \V(w_i)\cap \V(w_j), i\neq j\} |
\]
the number of such edges that contribute to more than one clique. Then
\[
\E E_1\le n^2m^2p^4
\]
and by Markov's inequality and \eqref{Eq:eGconc} we get that, for $mp^2\ll 1$ and $n^2mp^2\gg 1$, {\whp}
\begin{equation}\label{Eq:E1}
	E_1=O\left(\omega n^2m^2p^4\right)=O(\omega mp^2)e(\Gnmp).
\end{equation}
Now we may relate the number of edges $e(S)$, $e(S,\bar S)$, and the volume $\vol(S)$   in $\Gnmp$ to $X_{i,S}$ and $X_{i,\bar S}$. Recall that $X_{i,S}$ is the number of vertices in $S$ that chose attribute $w_i$. Therefore $\binom{X_{i,S}}{2}$ is the number of edges included in $S$  that are in the clique of vertices that chose $w_i$. Each edge from $S$ is included in at least one such clique. Similarly   $X_{i,S}X_{i,\bar S}$  is the number of edges with one end in $S$ and the other in $\bar S$ and included in the clique related to $w_i$. Therefore, taking into account that some edges might be included in two cliques (see  the  above discussion about $E_1$), we get 
\begin{equation}\label{Eq:DefeS}
	e(S)\le \sum_{i \in [m]} \binom{X_{i,S}}{2}=\frac{1}{2}\sum_{i \in [m]} (X_{i,S})_2.
\end{equation}
Moreover
\begin{equation}\label{Eq:DefeSbarS}
	e(S,\bar S)\ge  \sum_{i \in [m]} X_{i,S}X_{i,\bar S}-E_1,
\end{equation}
and
\begin{equation}\label{Eq:DefVolume}
	\begin{split}
		\vol(S)&=2e(S)+e(S,\bar S)\\
		&\ge \sum_{i \in [m]} X_{i,S}(X_{i,S}-1+X_{i,\bar S})-2E_1\\
		&= \sum_{i \in [m]} X_{i,S}(V_i-1)-2 E_1.
	\end{split}	
\end{equation}
\section{Auxiliary lemmas}\label{Sec:Auxiliary}
In this section we give some classical inequalities that are used in the remaining part of the article. 
We start with  Freedman's inequality. We state it here in the form that is convenient for our purposes (see corollary of Lemma 22 of \cite{Wa16} by Warnke or  \cite{BeDu22} by Bennett and Dudek).

\begin{lemma}  \label{lemma:Freedman}
	Let $M_0, M_1, \ldots, M_m$ be a martingale with respect to a filtration $\mathcal{F}_0 \subseteq \mathcal{F}_1 \subseteq \ldots \subseteq \mathcal{F}_n$. Let $|M_i - M_{i-1}|\le C$, for all $i\in [m]$, and $V= \sum_{i=1}^{m} \var(M_i-M_{i-1}|\mathcal{F}_{i-1}).$ Then 
	\[
	\Pra{M_m \geq M_0 + \lambda} \leq \exp\left\{-\frac{\lambda^2}{2V\left(1 + C\lambda/(3V)\right)}\right\}.
	\]
\end{lemma}
We will also use Chernoff's bound for the binomial distribution (see for example Theorem~2.1 and Corollary~2.3 in \cite{JLR_random_graphs})
\begin{lemma}  \label{lemma:Chernoff}
	Let $X$ has the binomial distribution $\Bin{n}{p}$, then for $t>0$
	\begin{equation*}
		\begin{split}
			\Pra{X\ge np + t}&\le \exp\left(-\frac{t^2}{2np+\frac{2}{3}t}\right);\\
			\Pra{X\ge np + t}&\le \exp\left(-\frac{t^2}{3np}\right), \text{ for }t/np\le 3/2;\\
			\Pra{X\le np - t}&\le \exp\left(-\frac{t^2}{2np}\right).\\
		\end{split}
	\end{equation*}
\end{lemma}

\section{Proof of Theorem \ref{Thm:Smallmp}}\label{Sec:Smallmp}
Recall that $mp^2\ll 1$ and $n^2mp^2\gg 1$. 
Denote by
$$\barV_i=\{v\in \V(w_i): |\W(v)|=1\},\quad \tilde{V}_i=|\barV_i|, i\in [m]$$
the sets of vertices that chose exactly one attribute $w_i$, $i\in [m]$, and their sizes. 
Note that $\tilde{V}_i$, $i\in [m]$, has the binomial distribution $\Bin{n}{p(1-p)^{m-1}}$, i.e. $\E(\tilde{V}_i)\sim npe^{-mp}$, $i\in [m]$. 
Set $\eps\in (0,1)$ and $ A_\eps = 3\eps^{-2}\ln (4/\eps^2)$. Let $np$ and $mp$ be such that $npe^{-mp}> A_\eps$ (i.e. also $np> A_\eps$). As $V_i$ and $\tilde V_i$, $i\in [m]$, have binomial distributions, by Chernoff's inequality (Lemma~\ref{lemma:Chernoff}), for all $i\in [m]$,
\begin{equation*}\begin{split}
	\Pra{\left|\tilde{V}_i-npe^{-mp}\right|\ge \eps npe^{-mp}}&\le 2\exp\left(-\frac{\eps^2 A_\eps}{3}\right)\le \frac{\eps^2}{2} \text{ and }\\
	\Pra{|V_i-np|\ge \eps np}&\le 2\exp\left(-\frac{\eps^2 A_\eps}{3}\right)\le \frac{\eps^2}{2}.
\end{split}\end{equation*}
Let
$$
\mathcal M_1=\left\{i\in [m]: |\tilde{V}_i-npe^{-mp}|\le \eps npe^{-mp} \text{ and }|V_i-np|\le \eps np\right\}.
$$
Then by the union bound and the above concentration results
\begin{multline}
	\E |[m]\setminus \mathcal M_1|\le\\ \le \sum_{i\in [m]}\Pra{ |\tilde{V}_i-npe^{-mp}|\ge \eps npe^{-mp} \text{ or }|V_i-np|\ge \eps np} \le  \eps^2m.
\end{multline}
Therefore by Markov's inequality $\Pra{|[m]\setminus \mathcal M_1|\ge \eps m}\le \eps,$ i.e.
\begin{equation}\label{Eq:M1upperbound}
	|\mathcal M_1|\ge (1-\eps)m\text{ with probability at least }1-\eps.
\end{equation}
Let us define a partition of the vertex set
$$
{\mathcal A} = \{\barV_i: i\in \mathcal M_1\}\cup \{\bar \V\},\text{ where } \bar \V=\V\setminus \bigcup_{i\in\mathcal M_1}\barV_i.
$$
Now set $G$, an instance of the random intersection graph $\Gnmp$ with properties 
\begin{equation}\label{Eq:GProperties}
	|\mathcal M_1|\ge (1-\eps)m\quad\text{ and }\quad|2e(G)-n^2mp^2|\le \eps n^2mp^2.
\end{equation}
Note that by  \eqref{Eq:eGconc} and \eqref{Eq:M1upperbound}, for any $\eps\in (0,1)$, $\Gnmp$ has the above stated properties  with probability at least $1-2\eps$.

First let us consider $S=\barV_i$, $i\in \mathcal M_1$. 
Note that by the definition of $\mathcal M_1$

\begin{equation*}\begin{split}
	2e(S)=2e(\barV_i)=\tilde V_i(\tilde V_i-1)
	&\ge (1-\eps)^2(npe^{-mp})^2\left(1-\frac{1}{(1-\eps)npe^{-mp}}\right)\\
	&\ge (1-\eps)^2(npe^{-mp})^2\left(1-\frac{1}{(1-\eps) A_\eps}\right)\\
	&\ge (1-\eps)^3(np)^2e^{-2mp}\\
	&\ge (1-\eps)^4\frac{2e(G)}{m}e^{-2mp}.
\end{split}\end{equation*}
The second last inequality follows as $(1-\eps) A_\eps>\eps^{-1}$ for all $\eps\in (0,0.8)$.
Moreover 
\begin{equation*}\begin{split}
	\vol (S) = \vol(\barV_i)
	= \tilde{V}_i(V_i-1)
	&\le (1+\eps)^2npe^{-mp}\cdot np \\
	&\le (1+\eps)^3\frac{\vol(G)}{m}e^{-mp}.
\end{split}\end{equation*}
Therefore, for $S=\barV_i$, $i\in \mathcal{M}_1$, and large enough  $m$ we get
\begin{equation}\label{Eq:Lower1}
	\begin{split}
		\frac{e(\barV_i)}{e(G)}-\left(\frac{\vol(\barV_i)}{\vol(G)}\right)^2&=\frac{e(S)}{e(G)}-\left(\frac{\vol(S)}{\vol(G)}\right)^2\\
		&\ge (1-\eps)^4\frac{e^{-2mp}}{m}-(1+\eps)^6\frac{e^{-2mp}}{m^2}
		\ge (1-5\eps)\frac{e^{-2mp}}{m}.
	\end{split}
\end{equation}
For $S=\bar \V$, using calculations for $\barV_i$, $i\in\mathcal M_i$, and $\eps<1/2$
\begin{equation*}
	\begin{split}
		e(\bar S)&=\sum_{i\in \mathcal M_1}e(\barV_i)\ge (1-\eps)m (1-\eps)^4\frac{e(G)e^{-2mp}}{m}\\
		&\ge (1-\eps)^5e(G)e^{-2mp}\ge (1-5\eps)e(G)e^{-2mp},\\
		(\vol(\bar S))^2
		&=\left(\sum_{i\in \mathcal M_1} \vol(\barV_i)\right)^2\\
		&\le (1+\eps)^6\vol^2(G)e^{-2mp}\le (1+21\eps)\vol^2(G)e^{-2mp}.
	\end{split}
\end{equation*}
Combining the above results with \eqref{Eq:ModularitySymmetry} we get for $\eps<1/2$ and large enough $m$
\begin{equation}\label{Eq:Lower2}
	\frac{e(\bar \V)}{e(G)}-\left(\frac{\vol(\bar \V)}{\vol(G)}\right)^2\ge - 26\eps e^{-2mp}.
\end{equation}
Recall that $\Gnmp$ with probability at least $1-2\eps$ has properties \eqref{Eq:GProperties}. As $G$ with properties \eqref{Eq:GProperties}
fulfil \eqref{Eq:Lower1} and \eqref{Eq:Lower2}, we get that  with probability at least $1-2\eps$  
\begin{multline*}
	\modul (\Gnmp)\ge \modul_{\mathcal A} (\Gnmp)\\
	\ge \sum_{i\in\mathcal M_i}(1-5\eps)\frac{e^{-2mp}}{m} - 26\eps e^{-2mp}\ge \left(1-31\eps\right) e^{-2mp}.
\end{multline*}

\section{Proof of Theorem \ref{Thm:Large_np_mp}}\label{Sec:Large_np_mp}
Recall the assumptions
\[
n^2mp^2\gg 1,\quad mp^2\ll 1,\quad m \gg n, \quad\text{ and }\quad np\gg \ln \frac{m}{n}.
\]
Set
\[
\eps=\eps_n=\sqrt{6\cdot \frac{\ln \left(\frac{3}{4}\cdot\frac{m}{n}\right)}{np}}=o(1).
\]
Set $S\subseteq \V$ such that $|S|=s\ge n/2$. Recall that and $X_i=X_{i,S}=|\V(w_i)\cap S|$, $i\in [m]$, is the number of vertices in $S$ that chose $w_i$. Moreover $X_i$  has the binomial distribution $\Bin{s}{p}$.
Let 
\[
\qq=\qq_s=\Pra{|X_i-sp|\ge \eps sp}.
\]
By Chernoff's inequality (see Lemma~\ref{lemma:Chernoff})
\[
\qq\le 2e^{-\frac{\eps^2}{3}sp}\le 2e^{-\frac{\eps^2}{6}np} \le 2e^{- \ln\left(\frac{3}{4}\cdot\frac{m}{n}\right)}=\frac{8}{3}\cdot\frac{n}{m}=:\delta.
\]
Let 
\[
\mathcal{M}_S=\{i\in [m]:\, |X_i-sp|\ge \eps sp\}, \quad M_S=|\mathcal{M}_S|.
\]
Then again by Chernoff's inequality (see Lemma~\ref{lemma:Chernoff})
\begin{multline*}
	\Pra{M_S\ge 2\delta m}
	\le \Pra{M_S\ge \qq m +\delta m} 
	\le \exp\left(-\frac{\delta^2m^2}{2\qq m + \frac23 \delta m}\right)\\
	\le \exp\left(-\frac38\delta m\right)=\exp(-n)=o(2^n).
\end{multline*}
Thus
\begin{equation}\label{Eq:MSduze}
	\Pra{\exists_{S,|S|\ge n/2} M_S\ge 2\delta m}
	\le \sum_{S,|S|\ge n/2}\Pra{M_S\ge 2\delta m}
	= o(1).
\end{equation}
In particular, when we set $S=\V$ we get that {\whp} $M_\V\le 2\delta m$. 
Moreover, again by Chernoff's inequality (Lemma~\ref{lemma:Chernoff}), for any $k\ge 5$
\[
\forall_{i\in [m]}\Pra{V_i\ge knp}\le \exp\left(-\frac{(k-1)^2}{2+2(k-1)/3}np\right) \le  e^{-2knp/3}.
\]
Set
\[
\mathcal{N}_k=\{i\in [m]: V_i\ge knp\},\quad N_k=|\mathcal{N}_k|,\quad k\ge 5.
\]
As $N_k$ has a binomial distribution with the expected value at most $me^{-2knp/3}$, by Markov's inequality, for $k\ge 5$,
\[
\Pra{N_k\ge me^{-knp/2}}\le  \frac{me^{-2knp/3}}{me^{-knp/2}}=e^{-knp/6}.
\]
Therefore, as $np\gg 1$,
\begin{equation}\label{Eq:Nkduze}
	\Pra{\exists_{k\ge 5} N_k\ge e^{-knp/2}m}\le \sum_{k=5}^{\infty}e^{-knp/6}=o(1).
\end{equation}
Let us define take an instance of $G$ such that 
\begin{multline}\label{Eq:DuzenpmpWlasnosci}
\forall_{S,|S|\ge n/2} M_S < 2\delta m, \quad
\forall_{k\ge 5} N_k < e^{-knp/2}m,\\
\text{and}\quad
e(\Gnmp)=n^2mp^2(1+O(\eps+\delta+\gamma)),\text{ where }\gamma=\omega mp^2.
\end{multline}
By \eqref{Eq:eGconc}, \eqref{Eq:MSduze}, and \eqref{Eq:Nkduze} we have that $\Gnmp$ has these properties. Then for $G$, an instance of $\Gnmp$, and for any $S$, $|S|=s\ge n/2$,
\begin{equation*}
	\begin{split}
		e(S)&=\sum_{i\in [m]}(X_i)_2\le\sum_{i\in [m]}X_i^2\\
		&\le
		\sum_{i\in [m]\setminus \mathcal{M}_S}(1+\eps)^2(sp)^2 
		+\sum_{i\in \mathcal{M}_S\setminus \mathcal{N}_5}(5np)^2
		+\sum_{k=6}^{\lfloor p^{-1}\rfloor}\sum_{i\in \mathcal{N}_{k-1}\setminus \mathcal{N}_{k}}(knp)^{2}\\
		&\le 
		m(sp)^2+3\eps m(sp)^2
		+2\delta m (5np)^2
		+\sum_{k=6}^{\infty}(knp)^2e^{-(k-1)np/2}m\\
		&\le s^2mp^2+O\left(\eps+\delta+(np)^2e^{-5np/2}\right)e(G)\\
		&= s^2mp^2+O\left(\left(\frac{\ln(m/n)}{np}\right)^{1/2}+\frac{n}{m}\right)e(G).
	\end{split}
\end{equation*}
And by \eqref{Eq:DefVolume}, as $np\eps\ge 1$,  
\begin{equation*}
	\begin{split}
		\vol(S)&\ge \sum_{i\in [m]}X_i(V_i-1)-E_1\\
		&\ge \sum_{i\in \mathcal{M}_S\cap\mathcal{M}_{\V}}(1-\eps)sp((1-\eps)np-1)-E_1\\
		&\ge \sum_{i\in \mathcal{M}_S\cap\mathcal{M}_{\V}}(1-\eps)sp((1-2\eps)np)-E_1\\
		&\ge \sum_{i\in \mathcal{M}_S\cap\mathcal{M}_{\V}}(1-3\eps)snp^2-E_1\\
		&\ge m(1-4\delta)(1-3\eps)snp^2-E_1\\
		&=snmp^2+O\left(\eps+\delta\right)e(G)-E_1\\
		&=snmp^2+O\left(\eps+\delta+\gamma\right)e(G).
	\end{split}
\end{equation*}
Combining the above estimates we get that for $G$ with properties \eqref{Eq:DuzenpmpWlasnosci}
\begin{equation*}
	\begin{split}
		&\frac{4e(S)e(G)-\vol^2(S)}{4e^2(G)}\\
		&=
		\frac{(s^2mp^2+O(\eps+\delta)e(G))\cdot n^2mp^2 (1+O(\eps+\delta+\gamma)) - (snmp^2+O(\eps+\delta)e(G))^2}{4e^2(G)}\\
		&=O\left(\eps+\delta+\gamma\right)\\
		&=O\left(\left(\frac{\ln(m/n)}{np}\right)^{1/2}+\frac{n}{m}+\omega mp^2\right),
	\end{split}
\end{equation*}
for any $\omega\to \infty$ arbitrarily slowly as $n\to\infty$.

\section{Proof of Theorem \ref{Thm:Smallnp}}\label{Sec:Smallnp}
Recall that we assume that $np=o(1)$, $mp^2=o(1)$, and $\dd=nmp^2\ge 1$.
Let us set $S$ such that $S=s\ge n/2$.
Recall that $V_i=|\V(w_i)|$ and $X_{i}=X_{i,S}$ have the binomial distributions $\Bin{n}{p}$ and $\Bin{s}{p}$, resp.. 
Moreover define $Y_{i}=X_{i,\bar S}\sim \Bin{n-s}{p}$, $i\in [m]$.  
By definition $X_1,\ldots,X_m,Y_1,\ldots,Y_m$ are all independent random variables.

We set a constant $A\ge 6$ and define auxiliary random variables
\begin{equation*}
	\tilde X_i = X_i\Ind_{X_i\le A}\quad \text{ and }\quad \tilde Y_i = Y_i\Ind_{Y_i\le A},\text{ for } i\in [m].   
\end{equation*}
For $l=1,2,3,4$ and $i\in [m]$
\begin{multline}\label{Eq:XiTildeExpectedValue}
	\E(\tilde X_i)_l=\E (X_i)_l\Ind_{X_i>A}\le \sum_{k=A}^{s}k^l\Pra{V_i\ge k}\le\sum_{k=A}^{s}k^l\binom{s}{k}p^k\le \sum_{k=A}^{s}k^l\left(\frac{esp}{k}\right)^k\\
	\le \sum_{k=A}^{s}\frac{(esp)^k}{k^{k-l}}
	\le A^l\left(\frac{esp}{A}\right)^A\frac{1}{1-esp}
	=(1+O(sp))A^l\left(\frac{esp}{A}\right)^A=O((sp)^A) 
\end{multline}
and similarly
\begin{equation*}
	\E(\tilde Y_i)_l=\E (Y_i)_l\Ind_{Y_i>A}=O(((n-s)p)^A). 
\end{equation*}
Let 
\begin{equation}\label{Eq:DefE2}
	E_2=\sum_{i\in [m]}(V_i)_2\Ind_{V_i>A}
\end{equation}
be the upper bound on number of edges contained in cliques of size larger than~$A$.
Substituting $s=n$, i.e. $S=\V$, to \eqref{Eq:XiTildeExpectedValue} we get
\begin{equation*}
	\E(E_2)=\E \left(\sum_{i\in [m]}(V_i)_2\Ind_{V_i>A}\right)= O\left(m(np)^A\right).
\end{equation*}
Therefore by Markov's inequality, for $B<A$
\begin{equation}\label{Eq:E2}
	\Pra{E_2\ge (np)^Bm}\le (np)^{A-B}=o(1).
\end{equation}

It is well known that
$\E (X_i)_l=(s)_lp^l$ and  $\E (Y_i)_l=(n-s)_lp^l$, $l=1,2,3,4$, $i\in [m]$. 
The observations made so far and standard calculations lead us to the following conclusions that for $A\ge 6$.

\begin{equation}	\label{Eq:ConditionalVariance}
		\begin{split}
		\E ((\tilde X_i)_2)&=s^2p^2-sp^2+O((np)^A),\\
		\E (\tilde X_i \tilde Y_i)&=s(n-s)p^2+O((np)^A),\\
		\var ((\tilde X_i)_2)&=2s^2p^2(1+O(np+n^{-1})),\\
		\var (\tilde X_i \tilde Y_i)&=s(n-s)p^2(1+O(np)).
	\end{split}
\end{equation}
We define two martingales, with the filtration $(\mathcal F_t)_{t\in [m]}$ defined as follows $\mathcal F_t=\sigma(\{X_i,Y_i:i\in [t]\})$
\begin{equation*}\begin{aligned}
		M_0&=M_0(S)=0, & M_t&=M_t(S)=\sum_{i=1}^{t}((\tilde X_i)_2-\E (\tilde X_i)_2),
		\\
		\bar M_0&=\bar  M_0(S)=0, &\bar M_t&=M_t(S)=\sum_{i=1}^{t}(\tilde X_i\tilde Y_i-\E (\tilde X_i\tilde Y_i)),
\end{aligned}\end{equation*}
for $t=1,2,\ldots,m$.
Then, by  \eqref{Eq:DefeS}, \eqref{Eq:DefeSbarS}, \eqref{Eq:E2},  {\whp}
\begin{equation}\label{Eq:eSVOLbyMartingale}
	\begin{split}
		2e(S)&\le \sum_{i=1}^{m}(\tilde X_i)_2 + E_2\\
		&=M_m+s^2mp^2+O((np)^B),
		\\
		\vol(S)&=2e(S)+e(S,\bar S)\\	 
		&\ge \sum_{i=1}^{m}(\tilde X_i)_2+\sum_{i=1}^{t}\tilde X_i\tilde Y_i - 2E_1\\
		&=M_m+s^2mp^2+\bar M_m+s(n-s)mp^2+O((np)^B)-2E_1\\
		&=M_m+\bar M_m+snmp^2 +O((np)^B)-2E_1.
\end{split}\end{equation}
Define  
\begin{equation*}
	\lambda_S=\eps_n snmp^2=\eps_n s\dd\quad\text{and}\quad \eps_n=\frac{3A^3}{\sqrt{\dd}},
\end{equation*}
where $\dd = nmp^2$.
We apply Lemma~\ref{lemma:Freedman} to both martingales. 
For $M_t$ by the definition of $\tilde X_i$ and \eqref{Eq:ConditionalVariance}
\begin{equation*}\begin{split}
		|M_i-M_{i-1}|&=|(\tilde X_i)_2-\E (\tilde X_i)_2|\le A^2;\\
		\var(M_i-M_{i-1}|\mathcal{F}_{i-1})&=\var((\tilde X_i)_2)\sim 2s^2p^2, \text{i.e.}\quad V\sim 2s^2mp^2.
\end{split}\end{equation*}
Then
\[
\frac{C\lambda_S}{3V}\le \frac{A^2\eps_n nsmp^2}{(1+o(1))6s^2mp^2}\le \frac12 A^2\eps_n, \text{ for }s\ge n/2.
\]
Then by Lemma~\ref{lemma:Freedman}, for $\dd\ge 1$, $A\ge 6$, and $s\ge n/2$, we have
\begin{equation*}\begin{split}
		\Pra{M_m\ge \lambda_S}&=\Pra{M_m\ge M_0 + \lambda_S}\\
		&\le \exp\left(-\frac{\eps_n^2s^2n^2m^2p^4}{4(1+o(1))s^2mp^2(1+A^2\eps_n/2)}\right)\\
		&\le \exp\left(-\frac{\eps_n^2\dd }{5(1+3A^5/2\sqrt{d})}n\right)\\
		&\le \exp\left(-\frac{9A^6}{5(1+3A^5/2)}n\right) \le \exp(-0.79n) = o(2^{-n}). 
\end{split}\end{equation*}
Moreover, as $(-M_t)_{t\in [m]}$ is also a martingale, in an analogous manner we get
\begin{equation*}\begin{split}
		\Pra{M_m\le -\lambda_S}&=\Pra{-M_m\ge -M_0 + \lambda_S}=o(2^{-n})
\end{split}\end{equation*}

Before we analyse the martingale $\bar M_t$, we consider the case $s\ge n-\eps_n n$. In this case 
$(n-s)smp^2\le \eps_n n smp^2\le \lambda_S$ and  trivially 
\begin{equation}\label{Eq:eSbarSduzeS}
	e(S,\bar S)\ge 0 \ge s(n-s)mp^2-\lambda_S, \text{ for }s\ge n-\eps_n n.
\end{equation}
Now we  use Lemma~\ref{lemma:Freedman} in the case $n/2\le s\le n-\eps_n n$. As before,  by \eqref{Eq:ConditionalVariance} for $\bar M_t$
\begin{equation*}\begin{split}
		|\bar M_i-\bar M_{i-1}|&=|\tilde X_i\tilde Y_i-\E \tilde X_i\tilde Y_i|\le A^2;\\
		\var(M_i-M_{i-1}|\mathcal{F}_{i-1})&=\var(\tilde X_i\tilde Y_i)\sim s(n-s)p^2, \text{i.e.}\quad V\sim s(n-s)mp^2.
\end{split}\end{equation*}
Then, for $n-s>\eps_nn$,
\[
\frac{C\lambda_S}{3V}\le \frac{A^2\eps_n nsmp^2}{(1+o(1))3s(n-s)mp^2}\le \frac12 A^2.
\]
Now, for $n/2\le s\le n-\eps_n n$, i.e. $n/(n-s)\ge 1$,
\begin{equation*}\begin{split}
		\Pra{\bar M_m\le -\lambda_S}&=\Pra{-\bar M_m\ge -\bar M_0 + \lambda_S}\
		\\
		&\le
		\exp\left(-\frac{\eps_n^2n^2s^2m^2p^4}{2(1+o(1))s(n-s)mp^2(1+A^2/2)}\right)\\
		&\le 
		\exp\left(-\frac{\eps_n^2\dd }{3(1+A^2/2)}n\right)\\
		&\le \exp\left(-\frac{3A^6}{(1+A^2/2)}\right)=o(2^{-n}). 
\end{split}\end{equation*}
Therefore
\begin{equation*}
	\Pra{\exists_{S, |S|\ge n/2}|M_S|\ge \lambda_S\text{ or }\bar M_S\le -\lambda_S}\le 2^n\cdot o(2^{-n})=o(1) 
\end{equation*}
and, as a consequence, {\whp}
\begin{equation}\label{Eq:MartingaleProperties}
	|M_S|\le \lambda_S \text{ and } \bar M_S\ge -\lambda_S.
\end{equation}
Let us take an instance $G$ of $\Gnmp$ with properties \eqref{Eq:MartingaleProperties} and \eqref{Eq:eGconc}. {\Whp} $\Gnmp$ has those properties.
Set $\lambda = \lambda_\V = \eps_n n^2mp^2=\eps_n n\dd\ge \lambda_S$. By \eqref{Eq:eSVOLbyMartingale} we have for $G$ and all $S\subseteq \V$
\begin{equation*}\begin{split}
		2e(S)&\le\lambda + s^2mp^2+O(\err),
		\\ 
		\vol S 
		&\ge -2\lambda +snmp^2 - O(\err),
\end{split}\end{equation*}
where we use a shorthand notation $\err:=(np)^B+\omega n^2m^2p^4$. 

By definition $\lambda\asymp n \sqrt{\dd} = n^{3/2}m^{1/2}p$ and $1\le \dd = nmp^2=o(n)$. Then, for $\omega$ tending slowly to infinity
\begin{equation*}
	nmp^2=\dd=o(\lambda),\quad
	\omega nm^{1/2}p=o(\lambda),\quad \text{and} \quad n^2m^2p^4\le \err.
\end{equation*}
Therefore by \eqref{Eq:eGconc} we have
\begin{equation}\label{Eq:eGconc2}
	2e(G)=n^2mp^2 + o(\lambda) + O(\err). 
\end{equation}
We recall that $\lambda = n^{3/2}m^{1/2}p = n^2mp^2/\sqrt{\dd}\le n^2mp^2$. Moreover $\err=o(n^2mp^2)$. Before we substitute the above estimates to the formula related to the modularity, we perform some calculations
\begin{equation*}\begin{split}
		2e(S)\cdot 2e(G)
		&\le (s^2mp^2+\lambda+O(\err))(n^2mp^2+o(\lambda)+O(\err))\\
		&= s^2n^2m^2p^4+\lambda(1+o(1)) n^2mp^2 +n^2mp^2 O(\err),\\
		(\vol S)^2&\ge
		(snmp^2-\lambda-O(\err))^2\\
		&=s^2n^2m^2p^4  + \lambda^2  - 2\lambda snmp^2 + (snmp^2+\lambda) O(\err)\\
		&\ge  s^2n^2m^2p^4 -  2\lambda n^2mp^2 - n^2mp^2O(\err).
\end{split}\end{equation*}
Thus we get for all $S\subseteq \V$
\begin{equation*}\begin{split}
		\frac{e(S)}{e(G)}-\frac{(\vol S)^2}{(\vol G)^2}
		&=
		\frac{2e(S)\cdot 2e(G)-(\vol S)^2}{4e(G)}\\
		&=(1+o(1))\frac{(1+o(1))3\lambda n^2mp^2 + n^2mp^2 O(\err)}{n^4m^2p^4}\\
		&=(1+o(1))\frac{9A^3}{\sqrt{\dd}}+O\left((np)^{B-2}+\omega mp^3\right).
\end{split}\end{equation*}
This combined with Corollary~\ref{corollary:OnlySbig} gives the statement of Theorem~\ref{Thm:Smallnp}. 

The proof of Corollary~\ref{Cor:Smallnp} is straight forward. It is enough to notice that $\dd=nmp^2\ll n^{2/3}$. Therefore $C/\sqrt{\dd}\gg n^{-1/3}$. On the other hand $\omega mp^2\ll n^{-1/3}$ for $\omega$ tending to infinity slowly enough. Moreover, when we set $A\ge 2$ then $(np)^A\le mp^2\ll n^{-1/3}$. Therefore $(np)^A+\omega mp^2 = o (1/\sqrt{\dd})$.

\section{Proof of Theorem \ref{Thm:LikeGnp}}\label{Sec:LikeGnp}

Let us take a random graph $\hatGnmp$ with the vertex set $\V$ and the edge set constructed in the following manner.  For each $i$ such that $V_i\ge 2$ in $\Gnmp$, we choose uniformly at random two vertices from $\V(w_i)$, in $\Gnmp$, and join them by an edge. This definition gives us a coupling of random graphs $(\Gnmp,\hatGnmp)$ such that with probability one 
$$
\hatGnmp\subseteq \Gnmp.
$$
The graph $\hatGnmp$ is very close to Erd\H{o}s-R\'enyi random graph  $\Gnp$ with 
$$
\barp = 1-e^{-m\hat q}, \text{ where }\hat{q}=1-(1-p)^n-np(1-p)^{n-1}.
$$ 
Namely
\begin{lemma}\label{Lem:Equivalence}
	$$
	{\rm d}_{TV}(G(n,\bar q),\hatGnmp)=o(1).
	$$
\end{lemma}
\begin{proof}
The proof relays on the same ideas as those used for example in lemmas~5 and~6 form \cite{Rybarczyk2011Equivalence}. 
Let 
$$
\hat E = |\{i\in [m]: V_i\ge 2\}|.
$$
Therefore
$$
\hat E \sim \Bin{m}{\hat{q}}, \quad \text{ for } \hat{q}=1-(1-p)^n-np(1-p)^{n-1}=\binom{n}{2}p^2(1+O(np)).
$$
Let us define also a random variable  with the Poisson distribution
$
\bar E \sim \Po{m\hat{q}}.
$ and let $\barGnmp$ be the graph on the vertex set $\V$ and with the edge set obtained by sampling with repetition $\bar E$ times edges from all two element subsets of $\V$ and merging multiple edges into one. A straight forward calculation shows that each edge in $\barGnmp$ appears independently with probability    
$$
\barp = 1-e^{-m\hat q/\binom{n}{2}}.
$$
Therefore repeating the reasoning from lemmas~5 and ~6 from~\cite{Rybarczyk2011Equivalence} we get
\begin{multline*}
	{\rm d}_{TV}(\Gnp,\hatGnmp)={\rm d}_{TV}(\barGnmp,\hatGnmp)\\
	\le
2{\rm d}_{TV}(\bar E,\hat E)\le \hat q = O((np)^2)=o(1).
\end{multline*}
Which ends the proof of the lemma.
\end{proof}

Now let us consider the probability space of the coupling $(\Gnmp,\hatGnmp)$ described above. Define a random variable
\begin{equation*} 
	\Delta=|E(\Gnmp)|-|E(\hatGnmp)|\le \sum_{i\in [m]} \binom{V_i}{2}-\Ind_{V_i\ge 2}.
\end{equation*}
Therefore, using inequality $(1-x)^n\le 1-nx+\binom{n}{2}x^2$, for $x\in (0,1)$

\begin{equation*}
\begin{split}
	\E \Delta 
	&= m\frac{(n)_2p}{2} - m(1-(1-p)^{n}-np(1-p)^{n-1})\le \frac{m(np)^3}{2}.
\end{split}
\end{equation*}
Thus by Markov's inequality, for any $\delta>0$, {\whp}
\begin{equation*}
	\Delta\le m(np)^{3-\delta}.
\end{equation*} 
Therefore,  for any $\delta>0$, {\whp}, for all sets $S\subseteq \V$ 
\begin{equation*}
\begin{split}
e_{\hatG}(S)&\le e_{\G}(S)\le e_{\hatG}(S)+m(np)^{3-\delta};\\
\vol_{\hatG}(S)&\le \vol_{\G}(S)\le \vol_{\hatG}(S)+2  m(np)^{3-\delta},
\end{split}
\end{equation*}
where we use a shorthand notation $\G=\Gnmp$ and $\hatG=\hatGnmp$.
In particular we have
$$
2e(\G)=\vol(\hatG)(1+O((np)^{1-\delta}))
$$
Therefore, for any $S\subseteq \V$, we get 
\begin{equation*}
	\begin{split}
		4e_{\G}(S)e(\G)-\vol^2_{\G}(S)
		&=
		4e_{\hatG}(S)e(\hatG)-\vol^2_{\hatG}(S)+O(m(np)^{3-\delta}\vol(\hatG))\\
		&=4e_{\hatG}(S)e(\hatG)-\vol^2_{\hatG}(S)+O((np)^{1-\delta})\vol^2(\hatG)		
	\end{split}
\end{equation*}
This implies that  {\whp}
 for any $S$
\begin{equation*}
\begin{split}
		\frac{e_{\G}(S)}{e(\G)}-\frac{(\vol_{\G} (S))^2}{(\vol (\G))^2}
		&=
		\frac
		{4e_{\G}(S)e(\G)-\vol^2_{\G}(S)}
		{(\vol (\G))^2}
		\\
		&=
		\frac
		{4e_{\hatG}(S)e(\hatG)-\vol^2_{\hatG}(S)+O((np)^{1-\delta})\vol^2(\hatG)}
		{\vol^2(\hatG)(1+O((np)^{1-\delta}))}
		\\
		&=\frac{e_{\hatG}(S)}{e(\hatG)}-\frac{(\vol_{\hatG} (S))^2}{\vol^2(\hatG)}+O((np)^{1-\delta}).
\end{split}
\end{equation*}
Therefore, for any partition $\mathcal{A}$ of $\V$, such that $|\mathcal{A}|\le k_n=\lfloor(np)^{-\delta}\rfloor$, {\whp}
\begin{equation*}	
\begin{split}
\modul_{\mathcal{A}}(\G)
&=\sum_{S\in \mathcal{A}}\left[\frac{e_{\G}(S)}{e(\G)}-\frac{(\vol_{\G} (S))^2}{(\vol (\G))^2}\right]\\
&=\sum_{S\in \mathcal{A}}\left[\frac{e_{\hatG}(S)}{e(\hatG)}-\frac{(\vol_{\hatG} (S))^2}{\vol^2(\hatG)}+O((np)^{1-\delta}).\right]\\
&=\modul_{\mathcal A}(\hatG) + O((np)^{1-2\delta}).
\end{split}\end{equation*}
Thus by Lemma~\ref{lemma:only2parts} {\whp}
\begin{equation*}
\begin{split}
\modul(\G) &\leq \frac{k_n}{k_n-1} \max_{\mathcal{A}: |\mathcal{A}| \leq k_n} \modul_{\mathcal{A}}(\G)\\
&=(1+o(1))\max_{\mathcal{A}: |\mathcal{A}| \leq k_n} \modul_{\mathcal{A}}(\hatG)+O((np)^{1-2\delta})\\
&\le (1+o(1))\modul(\hatG)+O((np)^{1-2\delta}).
\end{split}\end{equation*} 
In the same manner we get that  {\whp}
\begin{equation*}
\begin{split}
\frac{k_n-1}{k_n}\modul(\hatG)
&\le \max_{\mathcal{A}: |\mathcal{A}| \leq k_n} \modul_{\mathcal{A}}(\hatG)\\
&=\max_{\mathcal{A}: |\mathcal{A}| \leq k_n} \modul_{\mathcal{A}}(\G)+O((np)^{1-2\delta})\\
&\le \modul(\G)+O((np)^{1-2\delta}).
\end{split}\end{equation*}
Thus {\whp}
$$
\modul(\G)=(1+o(1))\modul(\hatG)+O((np)^{1-2\delta})
$$
This combined with Lemma~\ref{Lem:Equivalence} gives the result. 

\section{Concluding remarks}

We studied modularity of the classical random intersection graph introduced in \cite{KaronskiSchienerSinger1999Subgraph}. We showed that, when each vertex of the network does not have too many attributes and each attribute is possessed by many vertices, then there is an apparent community structure in the network detected by the modularity. We have also given an example in which there are still large communities related to common attributes, however the number of attributes of each vertex is large. In this case the community structure is no longer detected by the modularity, i.e. the modularity tends to $0$ as $n\to\infty$. It would be interesting to understand the behaviour of the random intersection graphs in the case when both $np$ and $mp$ tends to infinity, but $m=o(n)$. This case is not covered by the above mentioned results. We also studied a relation between the modularity of $\Gnmp$ and classical Erd\H{o}s--R\'eny random graph with independent edges.   

We think that these results give an insight into the relation between community structure of networks based on a bipartite graph, i.e. affiliation networks. We studied only the most classical random intersection graph model $\Gnmp$. We hope that the research will follow and more interesting models, form the point of view of applications, will be considered. Here we have in mind such models of random intersection graphs that not only have large clustering but also have power law degree distribution and possibly more properties of complex networks. 

%
%
%
\bibliographystyle{plain}

\begin{thebibliography}{10}
	
	\bibitem{Bagrow_trees_12}
	J.P. Bagrow.
	\newblock Communities and bottlenecks: Trees and treelike networks have high
	modularity.
	\newblock {\em Physical Review E}, 85:066118, 2012.
	
	\bibitem{Behrish2007PhaseTransition}
	M.~Behrisch.
	\newblock Component evolution in random intersection graphs.
	\newblock {\em The Electronic Journal of Combinatorics}, 14(1):R17, 2007.
	
	\bibitem{BeDu22}
	P.~Bennett and A.~Dudek.
	\newblock A gentle introduction to the differential equation method and dynamic
	concentration.
	\newblock {\em Discrete Mathematics}, 345(12):113071, 2022.
	
	\bibitem{BlGuLaLe08}
	V.D. Blondel, J.L. Guillaume, R.~Lambiotte, and E.~Lefebvre.
	\newblock Fast unfolding of communities in large networks.
	\newblock {\em Journal of Statistical Mechanics: Theory and Experiment},
	2008(10):P10008, 2008.
	
	\bibitem{Bloznelis2015SurveyModels}
	M.~Bloznelis, E.~Godehardt, J.~Jaworski, V.~Kurauskas, and K.~Rybarczyk.
	\newblock {\em Recent Progress in Complex Network Analysis: Models of Random
		Intersection Graphs}, pages 69--78.
	\newblock Springer Berlin Heidelberg, 2015.
	
	\bibitem{Bloznelis2015SurveyProperties}
	M.~Bloznelis, E.~Godehardt, J.~Jaworski, V.~Kurauskas, and K.~Rybarczyk.
	\newblock {\em Recent Progress in Complex Network Analysis: Properties of
		Random Intersection Graphs}, pages 79--88.
	\newblock Springer Berlin Heidelberg, 2015.
	
	\bibitem{BloznelisLaskela2021Superpositions}
	M.~Bloznelis, J.~Karjalainen, and L.~Leskelä.
	\newblock Assortativity and bidegree distributions on bernoulli random graph
	superpositions.
	\newblock {\em Probability in the Engineering and Informational Sciences},
	36(4):1188--1213, August 2021.
	
	\bibitem{BloznelisKaronski2014RIGProcess}
	M.~Bloznelis and M.~Karoński.
	\newblock Random intersection graph process.
	\newblock {\em Internet Mathematics}, 11(4–5):385--402, November 2014.
	
	\bibitem{BloznelisLeskela2023SuperpositionsClust}
	M.~Bloznelis and L.~Leskel\"a.
	\newblock Clustering and percolation on superpositions of bernoulli random
	graphs.
	\newblock {\em Random Structures \& Algorithms}, 63(2):283--342, February 2023.
	
	\bibitem{Brandes2008}
	U.~Brandes, D.~Delling, M.~Gaertler, R.~Gorke, M.~Hoefer, Z.~Nikoloski, and
	D.~Wagner.
	\newblock On modularity clustering.
	\newblock {\em IEEE Transactions on Knowledge and Data Engineering},
	20(2):172--188, February 2008.
	
	\bibitem{BrennanBresterNagaraj2020Equivalence}
	M.~Brennan, G.~Bresler, and D.~Nagaraj.
	\newblock Phase transitions for detecting latent geometry in random graphs.
	\newblock {\em Probability Theory and Related Fields}, 178(3–4):1215--1289,
	September 2020.
	
	\bibitem{ChFoSk21}
	J.~Chellig, N.~Fountoulakis, and F.~Skerman.
	\newblock The modularity of random graphs on the hyperbolic plane.
	\newblock {\em Journal of Complex Networks}, 10(1):cnab051, 2021.
	
	\bibitem{DiTh11}
	T.N. Dinh and M.T. Thai.
	\newblock Finding community structure with performance guarantees in scale-free
	networks.
	\newblock In {\em 2011 IEEE Third International Conference on Privacy,
		Security, Risk and Trust and 2011 IEEE Third International Conference on
		Social Computing}, pages 888--891, 2011.
	
	\bibitem{FillScheinSinger2000Equivalence}
	J.~A. Fill, E.~R. Scheinerman, and K.~B. Singer-Cohen.
	\newblock Random intersection graphs when $m=\omega(n)$: An equivalence theorem
	relating the evolution of the {G}(n, m, p) and {G}(n, p) models.
	\newblock {\em Random Structures \& Algorithms}, 16:156--176, 2000.
	
	\bibitem{GiSh_23}
	G.~Gilad and R.~Sharan.
	\newblock From {L}eiden to {T}el-{A}viv {U}niversity {(TAU)}: exploring
	clustering solutions via a genetic algorithm.
	\newblock {\em PNAS Nexus}, 2(6):pgad180, 2023.
	
	\bibitem{GodehardtJaworski2003TwoModels}
	E.~Godehardt and J.~Jaworski.
	\newblock Two models of random intersection graphs for classification.
	\newblock In O.~Opitz and M.~Schwaiger, editors, {\em Studies in Classifcation,
		Data Analysis and Knowledge Organization}, volume~22, pages 67--81. Springer,
	2003.
	
	\bibitem{GuillaumeLatapy2004Bipartite}
	J.-L. Guillaume and M.~Latapy.
	\newblock Bipartite structure of all complex networks.
	\newblock {\em Information Processing Letters}, 90(5):215--221, June 2004.
	
	\bibitem{JLR_random_graphs}
	S.~Janson, T.~{\L}uczak, and A.~Ruci{\'n}ski.
	\newblock {\em Random Graphs}.
	\newblock John Wiley $\&$ Sons, Inc., 2000.
	
	\bibitem{KaPrTh21}
	B.~Kami{\'n}ski, P.~Pra{\l}at, and F.Th{\'e}berge.
	\newblock {\em Mining Complex Networks}.
	\newblock Chapman and Hall/CRC, 2021.
	
	\bibitem{KaronskiSchienerSinger1999Subgraph}
	M.~Karo\'{n}ski, E.~R. Scheinerman, and K.B. Singer-Cohen.
	\newblock On random intersection graphs: The subgraph problem.
	\newblock {\em Combinatorics, Probability and Computing}, 8:131--159, 1999.
	
	\bibitem{KimLeeNa2018Equivalence}
	J.~H. Kim, S.~J. Lee, and J.~Na.
	\newblock On the total variation distance between the binomial random graph and
	the random intersection graph.
	\newblock {\em Random Structures \& Algorithms}, 52(4):662--679, January 2018.
	
	\bibitem{LagerasLindholm2008PhaseTransition2}
	A.~N. Lager\r{a}s and M.~Lindholm.
	\newblock A note on the component structure in random intersection graphs with
	tunable clustering.
	\newblock {\em The Electronic Journal of Combinatorics}, 15(1):N10, 2008.
	
	\bibitem{LaSu23}
	M.~Laso{\'n} and M.~Sulkowska.
	\newblock Modularity of minor-free graphs.
	\newblock {\em Journal of Graph Theory}, 102(4):728--736, 2023.
	
	\bibitem{Majstorovic2014}
	S.~Majstorovi\'c and D.~Stevanovi\'c.
	\newblock A note on graphs whose largest eigenvalues of the modularity matrix
	equals zero.
	\newblock {\em The Electronic Journal of Linear Algebra}, 27, January 2014.
	
	\bibitem{McSk_reg_lattice_13}
	C.~McDiarmid and F.~Skerman.
	\newblock Modularity in random regular graphs and lattices.
	\newblock {\em Electronic Notes in Discrete Mathematics}, 43:431--437, 2013.
	
	\bibitem{McDiSk2018RegTreelike}
	C.~McDiarmid and F.~Skerman.
	\newblock Modularity of regular and treelike graphs.
	\newblock {\em Journal of Complex Networks}, 6(4):596--619, 2018.
	
	\bibitem{McDiarmidSkerman2020}
	C.~McDiarmid and F.~Skerman.
	\newblock Modularity of {E}rdős‐{R}ényi random graphs.
	\newblock {\em Random Structures \& Algorithms}, 57(1):211--243, March 2020.
	
	\bibitem{NeBook10}
	M.E.J. Newman.
	\newblock {\em Networks: An introduction}.
	\newblock Oxford University Press, 2010.
	
	\bibitem{NeGi04}
	M.E.J. Newman and M.~Girvan.
	\newblock Finding and evaluating community structure in networks.
	\newblock {\em Physical Review E}, 69(2):026113, 2004.
	
	\bibitem{PrPrRa16_LNCS}
	L.~Prokhorenkova, P.~Pra{\l}at, and A.M. Raigorodskii.
	\newblock Modularity of complex networks models.
	\newblock In A.~Bonato, F.~Chung Graham, and P.~Pra{\l}at, editors, {\em
		Algorithms and Models for the Web Graph - 13th International Workshop, {WAW}
		2016, Montreal, QC, Canada, December 14-15, 2016, Proceedings}, volume 10088
	of {\em Lecture Notes in Computer Science}, pages 115--126, 2016.
	
	\bibitem{PrPrRa17}
	L.~Prokhorenkova, P.~Pra{\l}at, and A.M. Raigorodskii.
	\newblock Modularity in several random graph models.
	\newblock {\em Electronic Notes in Discrete Mathematics}, 61:947--953, 2017.
	
	\bibitem{Rybarczyk2011Equivalence}
	K.~Rybarczyk.
	\newblock Equivalence of the random intersection graph and {G}(n,p).
	\newblock {\em Random Structures \& Algorithms}, 38:205--234, 2011.
	
	\bibitem{Rybarczyk2011Coupling}
	K.~Rybarczyk.
	\newblock Sharp threshold functions for random intersection graphs via a
	coupling method.
	\newblock {\em The Electronic Journal of Combinatorics}, 18(1):P36, 2011.
	
	\bibitem{Rybarczyk2017coupling}
	K.~Rybarczyk.
	\newblock The coupling method for inhomogeneous random intersection graphs.
	\newblock {\em Electronic Journal of Combinatorics}, 24(2):P2.10, 2017.
	
	\bibitem{RybSulk2025+}
	K.~Rybarczyk and M.~Sulkowska.
	\newblock Modularity of preferential attachment graphs.
	\newblock preprint, arxiv.org/abs/2501.06771.
	
	\bibitem{Spirakis2013Survey}
	P.~Spirakis, S.~Nikoletseas, and Ch. Raptopoulos.
	\newblock A guided tour in random intersection graphs.
	\newblock In F.V. Fomin, R.~Freivalds, M.~Kwiatkowska, and D.~Peleg, editors,
	{\em Automata, Languages, and Programming. ICALP 2013.}, volume 7966 of {\em
		Lecture Notes in Computer Science}, pages 29--35, Springer, Berlin,
	Heidelberg, 2013.
	
	\bibitem{TrWaEc19}
	V.A. Traag, L.~Waltman, and N.J. van Eck.
	\newblock From {L}ouvain to {L}eiden: guaranteeing well-connected communities.
	\newblock {\em Scientific Reports}, 9(5233), 2019.
	
	\bibitem{HofstadKamjathyVadon2021RIG}
	R.~van~der Hofstad, J.~Komjáthy, and V.~Vadon.
	\newblock Phase transition in random intersection graphs with communities.
	\newblock {\em Random Structures \& Algorithms}, 60(3):406--461, December 2021.
	
	\bibitem{Wa16}
	L.~Warnke.
	\newblock On the method of typical bounded differences.
	\newblock {\em Combinatorics, Probability \& Computing}, 25:269--299, 2022.
	
\end{thebibliography}
%

\end{document}